\documentclass[12pt]{amsproc}

\usepackage{fancyhdr}
\author{Emerson de Melo}
\address{Department of Mathematics, University of Bras\'ilia, Bras\'ilia-DF 70910-900, Brazil}
\email{emerson@mat.unb.br}


\title{Nilpotent residual and Fitting subgroup of fixed points in finite groups}
\newtheorem{theorem}{\sc Theorem}[section]
\newtheorem{lemma}[theorem]{\sc Lemma}

\begin{document}


\begin{abstract}
Let $q$ be a prime and $A$ a finite $q$-group of exponent $q$ acting by automorphisms on a finite $q'$-group $G$. Assume that $A$ has order at least $q^3$. We show that if $\gamma_{\infty} (C_{G}(a))$ has order at most $m$ for any $a \in A^{\#}$, then the order of $\gamma_{\infty} (G)$ is bounded solely in terms of $m$. If the Fitting subgroup of $C_{G}(a)$ has index at most $m$ for any $a \in A^{\#}$, then the second Fitting subgroup of $G$ has index bounded solely in terms of $m$. 
\end{abstract}

\maketitle

\section{Introduction}

Suppose that a finite group $A$ acts by automorphisms on a finite group $G$.  The action is coprime if the groups $A$ and $G$ have coprime orders. We denote by $C_G(A)$ the set $$\{g\in G\ |\ g^a=g \ \textrm{for all} \ a\in A\},$$ the centralizer of $A$ in $G$ (the fixed-point subgroup). In what follows we denote by $A^\#$ the set of nontrivial elements of $A$. It has been known that centralizers of coprime automorphisms have strong influence on the structure of $G$.

Ward showed that if $A$ is an elementary abelian $q$-group of rank at least 3 and if $C_G(a)$ is nilpotent for any $a\in A^\#$, then the group $G$ is nilpotent \cite{War}. Later Shumyatsky showed that if, under these hypotheses, $C_G(a)$ is nilpotent of class at most $c$ for any $a\in A^\#$, then the group $G$ is nilpotent with $(c,q)$-bounded nilpotency class \cite{Sh}. Throughout the paper we use the expression ``$(a,b,\dots )$-bounded'' to abbreviate ``bounded from above in terms of  $a,b,\dots$ only''. In the recent article \cite{Eme1} the above result was extended to the case where $A$ is not necessarily abelian. Namely, it was shown that if $A$ is a finite group of prime exponent $q$ and order at least $q^3$ acting on a finite $q'$-group $G$ in such a manner that $C_G(a)$ is nilpotent of class at most $c$ for any $a\in A^{\#}$, then $G$ is nilpotent with class bounded solely in terms of $c$ and $q$. Many other results illustrating the influence of centralizers of automorphisms on the structure of $G$ can be found in \cite{khukhro}.

Recall that the nilpotent residual $\gamma_\infty(K)$ of a group $K$ is the last term of the lower central series of $K$. It can also be defined as the intersection of all normal subgroups of $K$ whose quotients are nilpotent. Also, recall that the second Fitting subgroup $F_2(G)$ of a finite group $G$ is defined as the inverse image of $F(G/F(G))$, that is, $F_2(G)/F(G)=F(G/F(G))$. Here $F(G)$ stands for the Fitting subgroup of $G$. 

Recently, in \cite{Eme11} it was proved that if $A$ is an elementary abelian group of order at least $q^3$ acting by automorphisms on a finite $q'$-group $G$ and if $|\gamma_{\infty}(C_{G}(a))|\leq m$ for any $a\in A^{\#}$, then the order of  $\gamma_{\infty}(G)$ is $m$-bounded. If $F(C_{G}(a))$ has index at most $m$ in $C_G(a)$ for any $a \in A^{\#}$, then the index of $F_2(G)$ is $m$-bounded.

In the present article we extend the results obtained in \cite{Eme11} to the case where $A$ is not necessarily abelian.  

We obtain the following results.

\begin{theorem}\label{main1}
Let $q$ be a prime and $m$ a positive integer. Let $A$ be a finite $q$-group of exponent $q$ acting by automorphisms on a finite $q'$-group $G$. Assume that $A$ has order at least $q^3$ and $|\gamma_{\infty} (C_{G}(a))| \leq m$ for any $a\in A^{\#}$. Then $|\gamma_{\infty}(G)|$ is $m$-bounded.   
\end{theorem}

\begin{theorem}\label{main2}
Let $q$ be a prime and $m$ a positive integer. Let $A$ be a finite $q$-group of exponent $q$ acting by automorphisms on a finite $q'$-group $G$. Assume that $A$ has order at least $q^3$ and $F(C_{G}(a))$ has index at most $m$ in $C_G(a)$ for any $a \in A^{\#}$. Then the index of $F_2(G)$ is $m$-bounded.   
\end{theorem}

In the next section we give some helpful lemmas that will be used in the proofs of the above results. Section 3 deals with the proof of Theorem \ref{main2}. In Section 4 we prove Theorem \ref{main1}.

\section{Preliminaries}
If $A$ is a group of automorphisms of a group $G$, the subgroup generated by elements of the form $g^{-1}g^\alpha$ with $g\in G$ and $\alpha\in A$ is denoted by $[G,A]$. The subgroup $[G,A]$ is an $A$-invariant normal subgroup in $G$. Our first lemma is a collection of well-known facts on coprime actions (see for example \cite{GO}). Throughout the paper we will use it without explicit references.
\begin{lemma}\label{111} Let $A$ be a group of automorphisms of a finite group $G$ such that $(|G|,|A|)=1$. Then
\begin{enumerate}
\item[i)] $G=[G,A]C_{G}(A)$.
\item[ii)] $[G,A,A]=[G,A]$.
\item[iii)] $A$ leaves invariant some Sylow $p$-subgroup of $G$ for each prime $p\in\pi(G)$.
\item[iv)] $C_{G/N}(A)=C_G(A)N/N$ for any $A$-invariant normal subgroup $N$ of $G$.
\item[v)] If $A$ is a noncyclic elementary abelian group and $A_1,\dots,A_s$ are the maximal subgroups in $A$, then $G=\langle C_G(A_1),\ldots,C_G(A_s)\rangle$. Furthermore, if $G$ is nilpotent, then $G=\prod_i C_G(A_i)$.
\end{enumerate}
\end{lemma}

Let $\varphi$ be an automorphism of prime order $q$ of a finite group $G$ and assume that $|C_G(\varphi)|=m$.  In \cite{Fong} Fong proved using the classification of finite simple groups that $G$ has a soluble subgroup of $(q,m)$-bounded index. Later Hartley and Meixner in \cite{Har} proved that if $G$ is soluble, then $G$ contains a nilpotent subgroup of $(q,m)$-bounded index. And lastly, it was shown by Khukhro that if $G$ is nilpotent, then $G$ has a subgroup of $(q,m)$-bounded index which is nilpotent of $q$-bounded class (see, for example, \cite[Theorem 5.3.1]{khukhro}). Combining this results we get the following theorem.  

\begin{theorem}\cite[Corollary 5.4.1]{khukhro}\label{afx}
If $\varphi$ is an automorphism of prime order $q$ of a finite group $G$ and $|C_G(\varphi)|=m$, then $G$ contains a nilpotent subgroup of $(q,m)$-bounded index whose nilpotency class is $q$-bounded. 
\end{theorem}

The following lemma is an application of the three subgroup lemma.
\begin{lemma}\label{lz}
Let  $A$ be a group of automorphisms of a finite group $G$ and let $N$ be a normal subgroup of $G$ contained in $C_G(A)$. Then $[[G,A],N]=1$. In particular, if $G=[G,A]$, then $N\leq Z(G)$.
\end{lemma}
\begin{proof} Indeed, by the hypotheses, $[N,G,A]=[A,N,G]=1$. Thus, $[G,A,N]=1$ and the lemma follows.
\end{proof}

We finish this section with some results about elementary abelian groups acting as automorphisms. 

\begin{lemma}\cite[Lemma 2.6]{Sh}\label{P2}
Let $q$ be a prime and $A$ be an elementary abelian group of order $q^2$ acting on a finite $q'$-group $G$. Let $C=\cap_{a\in A^{\#}}F(C_G(a))$. Then $C\leq F(G)$.
\end{lemma}

\begin{lemma}\cite[Lemma 3.1]{ASD}\label{P1}
Let $q$ be a prime and $m$ a positive integer. Let $A$ be an elementary abelian group of order $q^2$ acting on a finite $q'$-group $G$ in such a manner that $|C_G(a)|\leq m$ for any $a\in A^{\#}$. Then the order of $G$ is $m$-bounded.
\end{lemma}

We denote by $R(H)$ the soluble radical of a finite group $H$, that is, the largest normal soluble subgroup of $H$.

\begin{theorem}\cite[Theorem 2.5]{Eme11}\label{solu}
Let $q$ be a prime and $m$ a positive integer such that $m<q$. Let $A$ be an elementary abelian group of order $q^2$ acting on a finite $q'$-group $G$ in such a way that the index of $R(C_G(a))$ in $C_G(a)$ is at most $m$ for any $a\in A^{\#}$. Then $[G,A]$ is soluble.
\end{theorem}

\section{Proof of Theorem  \ref{main2}}

Let $F(G)$ denote the Fitting subgroup of a group $G$. Write $F_{0}(G)=1$ and let $F_{i+1}(G)$ be the inverse image of $F(G/F_{i}(G))$ for any $i=0,1,\ldots$. If $G$ is soluble, the least number $h$ such that $F_{h}(G)=G$ is called the Fitting height $h(G)$ of $G$.

Let us now assume the hypothesis of Theorem \ref{main2}. Thus, $A$ is a finite group of prime exponent $q$ and order at least $q^3$ acting on a finite $q'$-group $G$ in such a manner that $F(C_{G}(a))$ has index at most $m$ in $C_G(a)$ for any $a\in A^{\#}$. We wish to show that $F_2(G)$ has $m$-bounded index in $G$. If $A$ contains an abelian subgroup of order $p^3$, the result is immediate from \cite{Eme11}. Therefore without loss of generality we assume that all subgroups of order $p^3$ in $A$ are non-abelian. Clearly, $A$ must contain subgroups of order $p^3$. Thus, replacing if necessary $A$ by one of its subgroups of order $p^3$ we may assume that $A$ is an extra-special group of order $p^3$. As usual, we denote by  $A'$ the commutator subgroup of $A$. Since $A$ is extra-special, $A'=Z(A)$ and $|Z(A)|=p$.

First, using Theorem \ref{afx} we prove that $F_2(G)$ has $(q,m)$-bounded index in $G$. In what follows we denote by $B$ the subgroup $Z(A)$.  Note that $B$ is contained in all subgroups of order $p^2$ of $A$. Moreover, since $B=Z(A)$ we have that all centralizers $C_G(a)$ are $B$-invariants and $C_G(B)$ is $A$-invariant.

\begin{lemma}\label{F2}
The index of $F_2(G)$ is $(q,m)$-bounded.
\end{lemma}
\begin{proof}
Let $A_1, \ldots, A_{q+1}$ be the subgroups of order $q^2$ of $A$. Note that $C_G(B)$ admits the action of the quotient-group  $\overline{A}= A/B$ which is an elementary abelian group of order $q^2$. Thus, $C_G(B)=\langle C_{C_G(B)}(\overline{a}) \ | \ \overline{a}\in \overline{A}^{\#}\rangle$. An alternative way of expressing this is to say that $C_G(B)$ is generated by the subgroups $C_G(A_i)$ (see for example \cite[Lemma 2.3]{Eme1}) . 

Now, for any subgroup $A_i$ set $C_i=\cap_{a\in A_i^{\#}}F(C_G(a))$. The subgroup $C_i$ has $(q,m)$-bounded index in $C_G(A_i)$ and by Lemma \ref{P2} we have that $C_i\leq F(G)$. Let $\overline{G}=G/F(G)$. Then the image of each subgroup $C_G(A_i)$ in $\overline{G}$  has $(q,m)$-bounded order and so using Lemma \ref{P1} we obtain that $C_{\overline{G}}(B)$ has $(q,m)$-bounded order. Therefore by Theorem \ref{afx} $F(\overline{G})$ has $(q,m)$-bounded index and the proof is complete.  
\end{proof}

The above lemma says that if $q\leq m$, then $F_2(G)$ has $m$-bounded index. We will therefore assume that $q>m$. In this case,  $B$ acts trivially on $C_G(a)/F(C_G(a))$ for any $a\in A^{\#}$ and so $[C_G(a),B]\leq F(C_G(a))$ for any $a\in A^{\#}$. 

Now our main goal is to show that $[G,B]$ is nilpotent.

\begin{lemma}\label{L1}
Let $A_i$ be a subgroup of $A$ of order $q^2$ containing $B$ and suppose that $G=[G,B]$. Then $G$ is generated by the subgroups $[C_G(a),B]$ where $a\in A_i\setminus B$.
\end{lemma}
\begin{proof} Since $A_i$ is an elementary abelian group of order $p^2$ we have that $G=\langle C_G(a) \ | \ a\in A_i^{\#}\rangle$. If $G$ is abelian, the result is obvious since the subgroups $C_G(a)$ are $B$-invariants. If $G$ is nilpotent, the result can be seen by considering the action of $A$ on the abelian group $G/\Phi(G)$. Finally, the general case follows from the nilpotent case and Lemma 2.4 of \cite{Sh2}.
\end{proof}

\begin{lemma}\label{nilp}
The subgroup $[G,B]$ is nilpotent.
\end{lemma}
\begin{proof}
We argue by contradiction. By Lemma  \ref{solu}, the subgroup $[G,B]$ is soluble. Suppose $G$ is a counterexample of minimal possible order. Let $V$ be a minimal $A$-invariant normal subgroup of $G$. Then $V$ is an elementary abelian $p$-group and $G/V$ is an $r$-group for some primes $r\neq p$. Write $G=VH$ where $H$ is an $A$-invariant Sylow $r$-subgroup such that $H=[H,B]$. By the minimality we also have $V=[V,H]$, so that $C_V(H)=1$. 

Let $A_1, \ldots, A_{q+1}$ be the subgroups of order $q^2$ containing $B$. For each $a\in A^{\#}$ we denote by $V_a$ and $H_a$ the centralizers $C_V(a)$ and $C_H(a)$ respectively. It is clear that $V_a$ is normal in $C_G(a)$. Let $E_a=[H_a,B]$. Note that $V_aE_a \leq F(C_G(a))$ since $[C_G(a),B]\leq F(C_G(a))$. By Lemma \ref{L1} we have that  $H=\langle E_x \ | \ x\in A_i\setminus B \rangle$ for any $A_i$. Then $[C_V(A_i),H]=1$ for any $A_i$. Since $C_G(B)$ is generated by the subgroups $C_G(A_i)$ we obtain that $[C_V(B),H]=1$. So by the minimality of $G$ we have that $C_V(B)=1$. Note that $E_a$  centralizes $C_V(a)$ for any $a\in A^{\#}$ but there exists $a\in A^{\#}$ such that $E_a$ acts nontrivially on $V$. Our aim is a contradiction following from these assumptions.

Now, we regard $V$ as an $\mathbb{F}_pHA$-module and extend the ground field to a finite field $k$ that is a splitting field for $HA$. We now obtain a $kHA$-module $\widetilde{V} =V\otimes_{\mathbb{F}_p}k$. Many of the above-mentioned properties of $V$ are inherited by $\widetilde{V}$. In particular, $C_{\widetilde{V}}(H)=0$ and $E_a$ centralizes $C_{\widetilde{V}}(a)$ for any $a\in A^{\#}$.

Consider an unrefinable series of $kHA$-submodules 

\begin{equation*}
\widetilde{V}=V_1>V_2>\cdots V_n>V_{n+1}=0.
\end{equation*}

Let $W$ be one of the factors of this series. It is a nontrivial irreducible $kHA$-module. If $c\in H$ acts trivially on every such $W$, then $c$ acts trivially on $\widetilde{V}$, as the order of $H$ is coprime to the characteristic $p$ of the field $k$. Therefore without loss of generality we assume that $H$ does not centralizes $W$.

By Clifford's theorem \cite[Theorem 3.4.1]{GO}, $W$ is the direct sum
$$W=W_1\oplus \cdots \oplus W_t$$
of Wedderburn components $W_i$ with respect to $H$, which are $kH$-modules transitively permuted by $A$. Furthermore, on each of the $W_i$ the center of $H$ is represented by scalar multiplication. 

We denote by $\Omega$ the set of Wedderburn components $\{W_1,\ldots,W_t\}$. Since $H$ does not centralizes $W$ we can choose a Wedderburn component $W_1$ on which $H$ acts non-trivially. Thus, in order to get a contradiction  it is sufficient to prove that $H$ acts trivially on such $W_1$. 

\begin{lemma}\label{lemma reg orb}
Let $a\in A$. The subgroup $H$ acts trivially on the sum of components in any regular $\langle a \rangle$-orbit in $\Omega$.
\end{lemma}
\begin{proof}
Let $e\in E_a$ and consider the regular $\langle a \rangle$-orbit $\{W_j, W_j^{a},\ldots,W_j^{a^{q-1}}\}$. Note that $a\not\in B$ since $C_V(B)=1$. Let $A_1=\langle a, B \rangle$. Of course $|A_1|=q^2$.

The element $e$ leaves invariant all components $W_j^{a^i}$ since $e\in H$. On the other hand, $e$ fixes every element of the form $w_j+w_j^{a}+\cdots+w_j^{a^{q-1}}$ which belongs to $C_W(a)$. Hence $e$ acts trivially on all components $W_j^{a^i}$.

Now, note that for any $x\in A_1 \setminus B$ the $\langle x \rangle$-orbit $\{W_j, W_j^{x},\ldots,W_j^{x^{q-1}}\}$ is also regular since $C_V(B)=1$ (each component $W_j$ must be $B$-invariant). Hence, any element of $E_x$ acts trivially on $W_j$. Then $H$ acts trivially on all components $W_j^{a^i}$ since $H$ is generated by the subgroups $E_x$ such that $x\in A_1\setminus B$.
\end{proof}

By the above lemma we conclude that $W_1^a=W_1$ for any $a\in A^{\#}$. Recall that $A$ is not abelian and $B=A'=Z(A)$. Let $A_1$ be a subgroup of order $q^2$ containing $B$. The centralizers $C_G(x)$ such that $x\in A_1\setminus B$ are permuted by $A$ since $C_G(x)^a=C_G(x^a)$ and $x^a=x[x,a]$. Then  $C_{W_1}(a)\neq 0$ for any $a\in A_1\setminus B$ since $W_1=\langle C_{W_1}(x) \ | \ x\in A_1\setminus B\rangle$ and the subgroups $C_{W_1}(x)$ are permuted by $A$.

By the minimality we have that $[N,B]=1$ for any proper normal subgroup of $H$. Then $H$ is a special $r$-group, that is, $H$ is an elementary abelian group or has nilpotency class 2 and $H'=Z(H)=\Phi(H)$ is elementary abelian.

In order to prove that $H$ acts trivially on $W_1$ it suffices to prove that each subgroup $E_a$ acts trivially on $W_1$. Let $a\in A^{\#}$ such that $E_a$ does not act trivially on $W_1$. If $E_a$ is not abelian, then $E_a^{'} \leq Z(H)$ and acts trivially on $C_{W_1}(a)$. Therefore, $E_a^{'}$ acts trivially on $W_1$ since $Z(H)$ acts by scalar multiplication on $W_1$. In fact, if $E_a$ is not abelian, then $E_x$ is not abelian for any $x\in \langle a, B\rangle\setminus B$ since such subgroups are permuted by $A$. Moreover, since $C_{W_1}(x)\neq 0$ for any $x\in\langle a, B\rangle\setminus B$  we have that $Z=\langle E_x^{'} \ | \ x\in \langle a, B\rangle\setminus B \rangle $ acts trivially on $W_1$. Since $Z$ is a normal $A$-invariant subgroup of $H$, we can factor out $Z$ and assume that $E_a$ is abelian. But in this case $B$ acts fixed point freely on $W_1E_a$ and then $W_1E_a$ is a nilpotent group by Thompson's theorem \cite{T} which means that $E_a$ acts trivially on $W_1$ that is a contraction and the proof is complete.
\end{proof}

We can now easily complete the proof of Theorem \ref{main2}. By the above lemma $B$ acts trivially on the quotient $G/F(G)$. Therefore $G=F(G)C_G(B)$.
This shows that $F(C_G(B))\leq F_2(G)$. Since the index of $F(C_G(B))$ in $C_G(B)$ is at most $m$, the result follows.

\section{Proof of Theorem \ref{main1}}

We say that a finite group $G$ is metanilpotent if $\gamma_{\infty}(G)\leq F(G)$. The following elementary lemma will be useful (for the proof see for example \cite[Lemma 2.4]{AST}).

\begin{lemma}\label{I3} Let $G$ be a metanilpotent finite group. Let $P$ be a Sylow $p$-subgroup of $\gamma_{\infty} (G)$ and $H$ be a Hall $p'$-subgroup of G. Then $P=[P,H]$.
\end{lemma}

Let us now assume the hypothesis of Theorem \ref{main1}. Thus, $A$ is a finite group of prime exponent $q$ and order at least $q^3$ acting on a finite $q'$-group $G$ in such a manner that $\gamma_{\infty} (C_{G}(a))$ has order at most $m$ for any $a\in A^{\#}$. We wish to show that $\gamma_{\infty} (G)$ has $m$-bounded order. It is clear that $A$ contains a subgroup of order $q^3$. As in the previous section, we may assume that $A$ is an extra-special group of order $p^3$ and we denote by $B$ the subgroup $A'=Z(A)$.  Since $\gamma_{\infty}(C_{G}(a))$ has order at most $m$, we obtain that $F(C_G(a))$ has index at most $m!$ (see for example \cite[2.4.5]{khukhro}). By \cite[Theorem 1.1]{Eme2} $\gamma_{\infty} (G)$ has $(q,m)$-bounded order. Without loss of generality we will assume that $m!<q$. In particular, $[G,B]$ is nilpotent by Lemma \ref{nilp}.

\begin{lemma}\label{gam}
If $G$ is soluble, then $\gamma_{\infty}(G)=\gamma_{\infty}(C_G(B)) $.
\end{lemma}
\begin{proof}
We will use induction on the Fitting height $h$ of $G$. Suppose first that $G$ is metanilpotent. Let $P$ be a Sylow $p$-subgroup of $\gamma_{\infty}(G)$ and $H$ a Hall $A$-invariant $p'$-subgroup of $G$. By Lemma \ref{I3} we have  that $P=[P,H]$. It is sufficient to show that $P\leq \gamma_{\infty}(C_G(B))$. Therefore, without loss of generality, we assume that $G=PH$. With this in mind, observe that $\gamma_{\infty}(C_G(a))=[C_P(a),C_H(a)]$ for any $a\in A^{\#}$.

We will prove that $P=[C_P(B),C_H(B)]$. Let $A_1, \ldots, A_{q+1}$ be the subgroups of order $q^2$ of $A$. Note that $A_i$ acts trivially on $\gamma_{\infty}(C_G(a))$  for any $a\in A_{i}^{\#}$ since $q>m$. Hence $\gamma_{\infty}(C_G(a))\leq C_P(A_i)$ for any $a\in A_{i}^{\#}$. In particular, $ \gamma_{\infty}(C_G(B))\leq C_P(A)$ since $\gamma_{\infty}(C_G(B))$ is $A$-invariant.

First, assume that $P$ is abelian. Given $a,b\in A_i$ we have that  $[[C_P(a),C_H(a)], C_H(b)]\leq [C_P(A_i),C_H(b)]\leq [C_P(b),C_H(b)]$.  On the other hand, $H=\prod_{b\in A_i^{\#}} C_H(b)$. Therefore, the subgroup $N= \prod_{a\in A^{\#}}[C_P(a),C_H(a)]$ is a normal subgroup. Since $N$ is $A$-invariant, we obtain that $A$ acts on $G/N$ in such a way that $C_G(a)$ is nilpotent for any $a\in A^{\#}$. Thus $G/N$ is nilpotent by \cite[Lemma 2.5]{Eme2}. Therefore, $P=\prod_{a\in A^{\#}}[C_P(a),C_H(a)]$. In particular, $P=C_P(B)$.

Assume that $P$ is not abelian. Consider the action of $B$ on $P/\Phi(P)$. By the above, $P/\Phi(P)=C_P(B)\Phi(P)/\Phi(P)$. We see that $P=C_P(B)\Phi(P)$, which implies that $P=C_P(B)$.

Since $P=C_P(B)$ is a normal subgroup of $G$, we have that $[B,P,H]=[P,H,B]=1$. Then $[[H,B],P]=1$. Therefore, $P=[C_P(B),C_H(B)]$ since $H=C_H(B)[H,B]$.

If $G$ is soluble and its Fitting height is $h>2$, then we consider the quotient group $G/\gamma_{\infty} (F_2(G))$ which has Fitting height $h-1$ and $\gamma_{\infty} (F_2(G))\leq \gamma_{\infty}(G)$. Hence, by induction we have that $\gamma_{\infty}(G)=\gamma_{\infty}(C_G(B))$.

\end{proof}

Recall that under our assumptions $[G,B]$ is nilpotent and $C_G(B)$ has a normal nilpotent subgroup of index at most $m!$. Let $R$ be the soluble radical of $G$. Since $G=[G,B]C_G(B)$, the index of $R$ in $G$ is at most $m!$. Lemma \ref{gam} shows that the order of $\gamma_{\infty}(R)$ is at most $m$. We pass to the quotient $G/\gamma_{\infty}(R)$ and without loss of generality assume that $R$ is nilpotent. If $G=R$, we have nothing to prove. Therefore assume that $R<G$ and use induction on the index of $R$ in $G$. Since $[G,B]\leq R$, it follows that each subgroup of $G$ containing $R$ is $A$-invariant. If $T$ is any proper normal subgroup of $G$ containing $R$, by induction the order of $\gamma_{\infty}(T)$ is $m$-bounded and the theorem follows. Hence, we can assume that $G/R$ is a nonabelian simple group. We know that $G/R$ is isomorphic to a quotient of $C_G(B)$ and so, being simple, $G/R$ has order at most $m$.

As usual, given a set of primes $\pi$, we write $O_\pi(U)$ to denote the maximal normal $\pi$-subgroup of a finite group $U$. Let $\pi=\pi(m!)$ be the set of primes at most $m$. Let $N=O_{\pi'}(G)$. Our assumptions imply that $G/N$ is a $\pi$-group and $N\leq F(G)$. Thus,  by the Schur-Zassenhaus theorem \cite[Theorem 6.2.1]{GO} the group $G$ has an $A$-invariant $\pi$-subgroup $K$ such that $G=NK$. Let $K_0=O_\pi(G)$.

Suppose that $K_0=1$. Then $G$ is a semidirect product of $N$ by $K=C_K(A)$. For an automorphism $a\in A^\#$ observe that $[C_N(a),K]\leq\gamma_\infty(C_G(a))$ since $C_N(a)$ and $K$ have coprime order. On the one hand, being a subgroup of $\gamma_\infty(C_G(a))$, the subgroup $[C_N(a),K]$ must be a $\pi$-group. On the one hand, being a subgroup of $N$, the subgroup $[C_N(a),K]$ must be a $\pi'$-group. We conclude that $[C_N(a),K]=1$ for each $a\in A^\#$. Since $N$ is a product of all such centralizers $C_N(a)$, it follows that $[N,K]=1$. Since $K_0=1$ and $K$ is a $\pi$-group, we deduce  that $K=1$ and so $G=N$ is a nilpotent group.

In general $K_0$ does not have to be trivial. However considering the quotient $G/K_0$ and taking into account the above paragraph we deduce that $G=N\times K$. In particular, $\gamma_\infty(G)=\gamma_\infty(K)$ and so without loss of generality we can assume that $G$ is a $\pi$-group. It follows that the number of prime divisors of $|R|$ is $m$-bounded and we can use induction on this number. It will be convenient to prove our theorem first under the additional assumption that $G=G'$.

Suppose that $R$ is an $p$-group for some prime $p\in\pi$. Note that if $s$ is a prime different from $p$ and $H$ is an $A$-invariant Sylow $s$-subgroup of $G$, then in view of Lemma \ref{gam} we have $\gamma_{\infty}(RH)\leq\gamma_{\infty}(C_G(B))$ because $RH$ is soluble. We will require the following observation about finite simple groups (for the proof see for example \cite[Lemma 3.2]{Eme2}).

\begin{lemma}\label{uu}
Let $D$ be a nonabelian finite simple group and $p$ a prime. There exists a prime $s$ different from $p$ such that $D$ is generated by two Sylow $s$-subgroup.
\end{lemma}

In view of Lemma \ref{uu} and the fact that $G/R$ is simple we deduce that $G/R$ is generated by the image of two Sylow $s$-subgroup $H_1$ and $H_2$ where $s$ is a prime different from $p$. Both subgroups $RH_1$ and $RH_2$ are soluble and $B$-invariant since $[G,B]\leq R$. Therefore both $[R,H_1]$ and $[R,H_2]$ are contained in $\gamma_{\infty }(C_G(B))$. 

Let $H=\langle H_1,H_2\rangle$. Thus $G=RH$. Since $G=G'$, it is clear that $G=[R,H]H$ and $[R,G]=[R,H]$. We have $[R,H]=[R,H_1][R,H_2]$ and therefore the order of $[R,H]$ is $m$-bounded. Passing to the quotient $G/[R,G]$ we can assume that $R=Z(G)$. So we are in the situation where $G/Z(G)$ has order at most $m$. By a theorem of Schur the order of $G'$ is $m$-bounded as well (see for example \cite[2.4.1]{khukhro}). Taking into account that $G=G'$ we conclude that the order of $G$ is $m$-bounded.

Suppose now that $\pi(R)=\{p_1,\dots,p_t\}$, where $t\geq2$. For each $i=1,\dots,t$ consider the quotient $G/O_{p_i'}(G)$. The above paragraph shows that the order of $G/O_{p_i'}(G)$ is $m$-bounded. Since also $t$ is $m$-bounded, the result follows.

Thus, in the case where $G=G'$ the theorem is proved. Let us now deal with the case where $G\neq G'$. Let $G^{(l)}$ be the last term of the derived series of $G$. The previous paragraph shows that $|G^{(l)}|$ is $m$-bounded. Consequently, $|\gamma_{\infty}(G)|$ is $m$-bounded since $G/G^{(l)}$ is soluble and $G^{(l)}\leq\gamma_{\infty}(G)$. The proof is now complete.

\end{document}